\documentclass[11pt]{amsart}
\usepackage[numbers]{natbib}
\usepackage{amsmath,amsthm,amsfonts,amssymb}
\usepackage[czech,english]{babel}
\usepackage{graphicx}
\usepackage[all]{xy}
\usepackage[OT1]{fontenc}
\usepackage{algpseudocode}
\newtheorem{theorem}{Theorem}
\newtheorem{conjecture}{Conjecture}

\newtheorem{lemma}{Lemma}
\newtheorem{proposition}{Proposition}
\newtheorem{corollary}{Corollary}

\newtheorem{definition}{Definition}

\newtheorem{remark}{Remark}
\newcommand{\bbbn}{\mathbb{N}}

\DeclareGraphicsExtensions{.pdf}

\newcommand{\Land}{\mathbin{\&}}
\def\ato{tree-semilattice}
\def\mato{Borel {\ato}}

\newcommand{\ff}[2]{({#1})_{#2}}

\begin{document}
\title[Limits of Structures]{Limits of Structures and the  Example of Tree-Semilattices}
\author{Pierre Charbit}
\address{Pierre Charbit\\
LIAFA - Universit\'e Paris Diderot - Paris 7\\
Case 7014\\
75205 Paris Cedex 13, France}
\email{pierre.charbit@liafa.univ-paris-diderot.fr}

\author{Lucas Hosseini}
\address{Lucas Hosseini\\
Centre d'Analyse et de Math\'ematiques Sociales (CNRS, UMR 8557)\\
  190-198 avenue de France, 75013 Paris, France}
  \email{lucas.hosseini@gmail.com}
\thanks{Supported by grant ERCCZ LL-1201 
and CE-ITI P202/12/G061, and by the European Associated Laboratory ``Structures in
Combinatorics'' (LEA STRUCO)}

\author{Patrice Ossona~de~Mendez}
\address{Patrice~Ossona~de~Mendez\\
Centre d'Analyse et de Math\'ematiques Sociales (CNRS, UMR 8557)\\
  190-198 avenue de France, 75013 Paris, France
	--- and ---
Computer Science Institute of Charles University (IUUK)\\
   Malostransk\' e n\' am.25, 11800 Praha 1, Czech Republic}
\email{pom@ehess.fr}
\thanks{Supported by grant ERCCZ LL-1201 and by the European Associated Laboratory ``Structures in
Combinatorics'' (LEA STRUCO)}

\date{\today}



\begin{abstract}
The notion of left convergent sequences of graphs introduced by Lov\' asz et al. (in relation with homomorphism densities for fixed patterns and Szemer\'edi's regularity lemma) got increasingly studied over the past $10$ years. Recently, Ne\v set\v ril and Ossona de Mendez introduced a general framework for convergence of sequences of structures. In particular, the authors introduced the notion of $QF$-convergence, which is a natural generalization of left-convergence. In this paper, we initiate study of $QF$-convergence for structures with functional symbols by focusing on the particular case of tree semi-lattices. We fully characterize the limit objects and give an application to the study of left convergence of $m$-partite cographs, a generalization of cographs.

%
%

\end{abstract}

\maketitle
\section{Introduction}
The study of limits of graphs gained recently a major interest \cite{Borgs20081801,Borgs2012,pre05504139,LovaszBook,Lov'asz2006,Lovasz2010}.
In the framework studied in the aforementioned papers, 
a sequence $(G_n)_{n\in\bbbn}$ of graphs is said {\em left-convergent} if, for every (finite) graph $F$, the probability
$$t(F,G_n)=\frac{{\rm hom}(F,G_n)}{|G_n|^{|F|}}$$
that a random map $f:V(F)\rightarrow V(G_n)$ is a {\em homomorphism} (i.e. a mapping preserving adjacency) converges as $n$ goes to infinity.  (For a graph $G$, we denote by $|G|$ the {\em order} of $G$, that is the number of vertices of $G$.) In this case, the limit object can be represented by means of a {\em graphon}, that is
a measurable function $W:[0,1]^2\rightarrow [0,1]$. The definition of the 
function $t$ above is extended to graphons by
$$
t(F,W)=\idotsint \prod_{ij\in E(F)}W(x_i,x_j)\ {\rm d}x_1\,\dots\,{\rm d}x_p
$$
(where we assume that $F$ is a graph with vertex set $\{1,\dots,p\}$) and
then the graphon $W$ is the left-limit of a left-convergent sequence
of graphs $(G_n)_{n\in\bbbn}$ if for every graph $F$ it holds
$$
t(F,W)=\lim_{n\rightarrow\infty}t(F,G_n).$$

For $k$-regular hypergraphs, the notion of left-convergence extends in the natural way, and left-limits --- called {\em hypergraphons} --- are measurable functions  $W : [0, 1]^{2^k -2} \rightarrow [0, 1]$ and have been  constructed by Elek and Szegedy using ultraproducts \cite{ElekSze} (see also \cite{Zhao2014}). These limits were also studied by Hoover \cite{Hoover1979}, Aldous \cite{AldousICM}, and Kallenberg \cite{Kallenberg2005} in the setting of exchangeable random arrays (see also~\cite{Austin2008}). 
For other structures, let us mention limits of permutations \cite{hoppen2011limits, Hoppen201393} and limits of posets \cite{brightwell2010continuum, janson2011poset,hladky2015poset}.

A {\em signature} $\sigma$ is a set of symbols of relations and functions with their arities. A $\sigma$-structure $\mathbf A$ is defined by its {\em domain} $A$ and an interpretation in $A$ of all the relations and functions declared in $\sigma$.  A $\sigma$-structure is {\em relational} if the signature $\sigma$ only contains symbols of relations. Thus relational structures are natural generalization of $k$-uniform hypergraphs. To the opposite, a $\sigma$-structure is {\em functional} (or called an {\em algebra})  if the signature $\sigma$ only contains symbols of functions. Denote by ${\rm QF}_p(\sigma)$ the fragment of all quantifier free formulas with $p$ free variables (in the language of $\sigma$) and by ${\rm QF}(\sigma)=\bigcup_p {\rm QF}_p(\sigma)$ the fragment of all quantifier free formulas. 
In the following, we shall use
${\rm QF}_p$ and ${\rm QF}$ when the signature $\sigma$ is clear from context. For a formula $\phi$ with $p$ free variables, the set of satisfying assignments of $\phi$ is denoted by $\phi(\mathbf A)$:
$$\phi(\mathbf A)=\{(v_1,\dots,v_p)\in A^p:\ \mathbf A\models \phi(v_1,\dots,v_p)\}$$

In the general framework of finite $\sigma$-structures (that is a $\sigma$-structure with finite domain), the notion of $QF$-convergence has been introduced by Ne\v set\v ril and the third author \cite{CMUC}. In this setting, a sequence $(\mathbf A_n)_{n\in\bbbn}$ of $\sigma$-structures is {\em $QF$-convergent} if, for every quantifier free formula $\phi$ with free variables $x_1,\dots,x_p$, the probability
\begin{equation}
\label{eq:SP}
\langle\phi,\mathbf A_n\rangle=\frac{|\phi(\mathbf A_n)|}{|A_n|^p}
\end{equation}
that a random (uniform independent) assignment to the free variables of $\phi$ of elements of $A_n$ satisfies $\phi$ converges as $n$ goes to infinity. These notions naturally extends to {\em weighted structures}, that is structures equipped with a non uniform probability measure.

The notion of QF-convergence extends several notions of convergence. 

It was proven in \cite{CMUC} that a sequence of graphs (or of $k$-uniform hypergraphs) with order going to infinity is QF-convergent if and only if it is left-convergent. This is intuitive, as for every finite graph $F$ with vertex set $\{1,\dots,p\}$ there is a quantifier-free formula $\phi_F$ with free variable $x_1,\dots,x_p$ such that for every graph $G$ and every $p$-tuple $(v_1,\dots,v_p)$ of vertices of $G$ it holds
$G\models\phi_F(v_1,\dots,v_p)$ if and only if the map
$i\mapsto v_i$ is a homomorphism from $F$ to $G$.

As mentioned before the left-limit of a left-convergent sequence of graphs can be represented by a graphon. However it cannot, in general, be represented by a Borel graph ---  that is a graph having a standard Borel space $V$ as its vertex set and a Borel subset of $V\times V$ as its edge set.
A graphon $W$ is {\em random-free} if it is almost everywhere $\{0,1\}$-valued.
Notice that a random-free graphon is essentially the same (up to isomorphism mod $0$) as a Borel graph equipped with a non-atomic probability measure on $V$.
A class of graph $\mathcal C$ is said to be {\em random-free} if, for every
left-convergent sequence of graphs $(G_n)_{n\in\bbbn}$ with $G_n\in\mathcal C$ (for all $n$) the sequence $(G_n)_{n\in\bbbn}$ has a random-free limit. 

Local convergence of graphs with bounded degree has been defined by Benjamini and Schramm~\cite{Benjamini2001}. A sequence $(G_n)_{n\in\bbbn}$ of graphs with maximum degree $D$ is {\em local-convergent} if, for every $r\in\bbbn$, the distribution of the isomorphism types of the distance $r$-neighborhood of a random vertex of $G_n$ converges as $n$ goes to infinity. 
This notion can also be expressed by means of QF-convergence (in a slightly stronger form).
Let $G_1,\dots,G_n,\dots$ be graphs with maximum degree strictly smaller than $D$. By considering a proper edge coloring of $G_n$ by $D$ colors, we can represent $G_n$ as a functional structure $\mathbf V_n$ with signature containing $D$ unary functions $f_1,\dots,f_D$, where $V_n$ is the vertex set of $G_n$ and $f_1,\dots,f_D$ are defined as follows: for every vertex $v\in V_n$,
$f_i(v)$ is either the unique vertex adjacent to $v$ by an edge of color $i$, or $v$ if no edge of color $i$ is incident to $v$. It is easily checked that if the sequence $(\mathbf V_n)_{n\in\bbbn}$ is 
QF-convergent if and only if  the sequence $(G_n)_{n\in\bbbn}$ of edge-colored graphs is local-convergent. 
If $(\mathbf V_n)_{n\in\bbbn}$ is QF-convergent, then the limit is a {\em graphing}, that is a functional structure $\mathbf V$ (with same signature as $\mathbf{V}_n$) such that
$V$ is a standard Borel space, and $f_1,\dots,f_D$ are measure-preserving involutions. 

In the case above, the property of the functions to be involutions is essential. The case of quantifier free limits of general functional structures is open, even in the case of unary functions. Only the simplest case of a single unary function has been recently settled \cite{MapLim}. The case of QF-limits of functional structures with a single binary function is obviously at least as complicated as the case of graphs, as a graph $G$ can be encoded by means of a (non-symmetric) function $f$  defined by $f(u,v)=u$ if $u$ and $v$ are adjacent, and $f(u,v)=v$ otherwise, with the property that QF-convergence of the encoding is equivalent to left-convergence of the graphs. The natural guess here for a limit object is the following:
\begin{conjecture}
Let $\sigma$ be the signature formed by a single binary functional symbol $f$.

Then the limit of a QF-convergent sequence of finite $\sigma$-structures can be represented by means of a measurable function
$w:[0,1]\times[0,1]\rightarrow \mathfrak{P}([0,1])$, where $\mathfrak{P}([0,1])$ stands for the space of probability measures on $[0,1]$.
\end{conjecture}

As witnessed by the case of local-convergence of graphs with bounded degrees, the ``random-free'' case, that is the case where the limit object can be represented by a Borel structure with same signature, is of particular interest. In this paper, we will focus on the case of simple structures defined by a single binary function  --- the tree semi-lattices --- and we will prove that they admit Borel tree semi-lattices for QF-limits. Conversely, we will prove that every Borel tree semi-lattices (with domain equipped with an atomless probability measure)
can be arbitrarily well approximated by a finite tree semi-lattices, hence leading to a full characterization of QF-limits of finite tree semi-lattices.

\section{Statement of the Results}
\label{sec:res}
  A {\em \ato{}} is an algebraic structure $\mathbf{T}
  = (T, \wedge)$ such that:
  \begin{enumerate}
  \item $(T, \wedge)$ is a meet semi-lattice (i.e. an idempotent
    commutative semigroup);
  \item $\forall x, y, z$ s.t. $x \wedge z = x$ and $y \wedge z = y$
    it holds $x \wedge y \in \{x, y\}$.
  \end{enumerate}

Because we consider structures with infimum operator $\wedge$, note that we shall use the symbol $\&$ for the logical conjunction.

Each {\ato} $\mathbf T$ canonically defines a partial order on its domain by
$x\leq y$ if $x\wedge y=x$. In the case where $T$ is finite, it is a partial order induced by the ancestor relation of a rooted  tree.

It is possible to add finitely many unary relations $M_1, \dots, M_k$ to the signature of {\ato}s. In this case, we speak of {\em colored} \ato{}s, and we define the color of a vertex $v$ as the set of the indices of those unary relations it belongs to: $c(v)=\{i: M_i(v)\}$.

  A {\em \mato} is a {\ato} $\mathbf{T}$ on a standard Borel space $T$, such that $\wedge:T\times T\rightarrow T$ is Borel.
Note that every finite {\ato} is indeed a {\mato}.  

Our main results concerning QF-convergence of {\ato}s are as follows:
\begin{itemize}
\item every ${\rm QF}$-convergent sequence of finite colored weighted {\ato}s
admits a limit, which is a Borel colored {\ato} (Theorem~\ref{thm:limwato}), and conversely: every 
Borel colored {\ato} is the limit of some  ${\rm QF}$-convergent sequence of colored weighted {\ato}s (Corollary~\ref{cor:wsamp});
\item every ${\rm QF}$-convergent sequence of colored uniform {\ato}s
admits a limit, which is an atomless Borel colored {\ato} (Corollary~\ref{cor:limato}), and conversely: every atomless
Borel colored {\ato} is the limit of some finite ${\rm QF}$-convergent sequence of colored uniform {\ato}s (Theorem~\ref{thm:uapprox}).
\end{itemize}

The notion of $m$-partite cographs has been introduced in \cite{Ganian2013}, based on the generalization of the characterization of cographs by means of cotrees \cite{Corneil1981}: a graph $G$ is an {\em $m$-partite cograph} if there exists a colored {\ato} $\mathbf{T}$, such that the vertices of $G$ are the leaves of $\mathbf T$, the leaves of $\mathbf{T}$ are colored with a set of at most $m$ colors,
 and the adjacency of any two vertices $x$ and $y$ is fully determined by the colors of $x,y$ and $x\wedge y$. (Notice that there is no restriction on the colors used for internal elements of $\mathbf{T}$.) In this setting we prove (Theorem~\ref{thm:limco}):

\begin{itemize}
\item every left-convergent sequence of $m$-partite cographs has a Borel limit, which is the interpretation of an atomless Borel colored {\ato};
\item conversely, every interpretation of an atomless Borel colored {\ato} is the left-limit of a sequence of $m$-partite cographs.
\end{itemize}

The class of all finite $m$-partite cographs can be characterized by means of a finite family $\mathcal{F}_m$ of excluded induced subgraphs \cite{Ganian2012,Ganian2013}. We prove that this characterization extends to Borel graphs (Theorem~\ref{thm:char}) in the sense that 
an atomless Borel graph excludes all graphs in $\mathcal{F}_m$ as induced subgraphs if and only if it is the interpretation of an atomless colored {\mato}.

\section{Preliminaries}

%
%
In a finite \ato{}, each element $x$ except the minimum has a unique predecessor, that we call the {\em father} of $x$ (as it is the father of $x$ in the associated tree).

For a {\ato} $\mathbf T$ and an element $v\in T$ we further define
\begin{align*}
 T_v &= \{u \in T; u \geq v\}\\
F_v &=\{u\text{ minimal in } T_v\setminus\{v\}\}.
\end{align*}

  Let $\mathbf{T}$ be a \ato{}, and let $u_1, \dots, u_p \in T$. The {\em sub-{\ato}}
   $\mathbf{T}\langle u_1, \dots, u_p\rangle$  of $\mathbf T$ {\em generated} by $u_1,\dots,u_p$ is the {\ato} with elements
  $$\Bigl\{ \bigwedge_{i \in I} u_i ; \quad \emptyset \neq I \subseteq
  \{1, \dots, p\} \Bigr\},$$
  where $\wedge$ is defined as the restriction of the $\wedge$ function
  of $\mathbf T$  to the domain of  $\mathbf{T}\langle u_1, \dots, u_p\rangle$. 
  \begin{remark}
  \label{rem:gato}
  Condition (2) of the definition of a {\ato} can be replaced by condition:
  \begin{equation*}
  \forall x,y,z\in T\qquad |\{x\wedge y, x\wedge z,y \wedge z\}|\leq 2.
  \end{equation*}
     \end{remark}

 It follows that for $u_1, \dots, u_p \in T$, the sub-{\ato}
   $\mathbf{T}\langle u_1, \dots, u_p\rangle$  of $\mathbf T$  generated by $u_1,\dots,u_p$ has domain
   $$
   T\cup\{u_i\wedge u_j: 1\leq i< j\leq p\}.
   $$
   

If $\phi\in{\rm QF}_p$ is any quantifier-free formula with $p$ free variables (in the language of {\ato}) and $\mathbf{T}$ is a {\mato} then $\phi(\mathbf T)$ is a Borel subset of $T^p$, thus any (Borel) probability measure $\mu$ on $\mathbf{T}$ allows to define 
 $$\langle \phi,(\mathbf{T},\mu)\rangle=\mu^{\otimes p}(\phi(\mathbf{T})).$$

  Let $\mathbf{T}_1,\mathbf{T}_2$ be \mato{}s, and let $\mu_1$ and $\mu_2$ be probability measures on $T_1$ and $T_2$, respectively.
We define  the pseudometric
$${\rm dist}((\mathbf{T}_1, \mu_1), (\mathbf{T}_2, \mu_2)) =
  \sum_{p \geq 1} 2^{-p}\,  \sup_{\phi\in{\rm QF}_p} \left| \langle
  \varphi_p, (\mathbf{T}_1, \mu_1) \rangle - \langle \varphi_p, (\mathbf{T}_2,
  \mu_2) \rangle \right|.$$
Note that a sequence of {\mato}s
is  ${\rm QF}$-convergent if and only if it is Cauchy for the above distance.

As mentioned in Section~\ref{sec:res}, 
the color of an element $v$ of a  colored \ato{} is the set $c(v)$ of the indices of those unary relations $v$ belongs to.
 The order of the relations naturally induces a total order on these colors:
 for distinct  $c_1,c_2\subseteq \{1,\dots,k\}$ it holds $c_1<c_2$ if
 $\min(c_1\setminus c_2) <\min(c_2\setminus c_1)$ (with convention
 $\min\emptyset=0$).

%

\section{Sampling and Approximating}
%
Two Borel structures $\mathbf{A},\mathbf{B}$ are {\em QF-equivalent} if 
$\langle\phi,\mathbf{A}\rangle=\langle\phi,\mathbf{B}\rangle$ holds for every quantifier free formula $\phi$.
The following lemma, which is trivial for uniform structures, requires some little work for structures with a probability measure. As it this result is not really needed here, we 
leave the proof to the reader.
\begin{lemma}
Two finite structures are QF-equivalent if and only if they are isomorphic.
\end{lemma}
%

\begin{definition}
Let $\mathbf{A}$ be a Borel structure, and let $n,L\in\bbbn$.
The {\em $(n,L)$-sampling} of $\mathbf{A}$ is the random structure $\mathbf{B}_{X_1,\dots,X_n}$ defined as follows:
\begin{itemize}
  \item the domain of $\mathbf{B}_{X_1,\dots,X_n}$ is the union of sets $\Omega, \Gamma$, and $\{\bot\}$, where $\Omega=\{X_1,\dots,X_n\}$ is a set of $n$ random independent elements of $\mathbf{A}$ sampled with probability $\mu_{\mathbf{A}}$,  $\Gamma$ is the set of all the elements that can be obtained from  $\Omega$ by at most $L$ applications of a function, and $\bot$ is an additional element;
  \item the relations  are defined on $\mathbf{B}_{X_1,\dots,X_n}$ as in $\mathbf{A}$, as well as functions when the image belongs to $\Omega\cup\Gamma$.
  When undefined, functions have image $\bot$;
  \item the probability measure on $\mathbf{B}_{X_1,\dots,X_n}$ assigns probability $1/n$ to $X_1,\dots,X_n$, and probability $0$ to the other elements.
\end{itemize}
\end{definition}

\begin{lemma}[\cite{mcdiarmid1989method}]
\label{lem:mcd}
Let $X_1,\dots,X_n$ be independent random variables, with $X_k$ taking values in a set $A_k$ for each $k$. Suppose that a (measurable) function $f:\prod A_k\rightarrow\mathbb{R}$ satisfies
$$|f(\mathbf{x})-f(\mathbf{x}')|\leq c_k$$
whenever the vectors $\mathbf{x}$ and $\mathbf{x}'$ differ only in the $k$th coordinate. Let $Y$ be the random variable $f[X_1,\dots,X_n]$. Then for any $t>0$,
$${\rm Pr}(|Y-\mathbb{E}Y|\geq t)\leq 2e^{-2t^2/\sum c_k^2}.$$
\end{lemma}

\begin{lemma}
\label{lem:samp}
Let $\phi$ be a quantifier-free formula $\phi$ with at most $p$ free variables and at most $L$ functional symbols.

Then, for every Borel structure $\mathbf{A}$ and every $\epsilon>0$ and $n>2/\epsilon$, it holds
\begin{equation*}
{\rm Pr}(|\langle\phi,\mathbf{A}\rangle - \langle\phi,\mathbf{B}_{X_1,\dots,X_n}\rangle|\geq \epsilon)\leq
2e^{-\frac{(\epsilon n-2)^2}{p^2n}},
\end{equation*}
where  
$\mathbf{B}_{X_1,\dots,X_n}$ is the $(n,L)$-sampling of $\mathbf{A}$.
\end{lemma}
\begin{proof}
Let   $n\geq p$ and let
$f$ be the indicator function of $\phi(\mathbf{A})$. Then
\begin{align*}
\langle\phi,\mathbf{A}\rangle&=\idotsint f(x_1,\dots,x_p)\,{\rm d}x_1\dots{\rm d}x_p\\
&=\frac{1}{\ff{n}{p}}\idotsint \sum_{g\in {\rm Inj}([p],[n])}f(x_{g(1)},\dots,x_{g(p)})\,{\rm d}x_1\dots{\rm d}x_n,\\
\end{align*}
where $\ff{n}{p}=n(n-1)\dots(n-p+1)$ is the Pochhammer symbol.

Let
$$h(x_1,\dots,x_n)=\frac{1}{n^p}\sum_{g\in [n]^{[p]}} f(x_{g(1)},
\dots,x_{g(p)}).$$

Then it holds
$$\langle{\phi},\mathbf{B}_{X_1,\dots,X_n}\rangle=h(X_1,\dots,X_n).$$

Considering the expectation  we get
$$
\mathbb{E}h(X_1,\dots,X_n)=\frac{1}{n^p}\idotsint\sum_{g\in [n]^{[p]}} f(x_{g(1)},\dots,x_{g(p)})\,{\rm d}x_1\dots{\rm d}x_n.
$$
So we have
\begin{align*}
|\langle\phi,\mathbf{A}\rangle-\mathbb{E}h(X_1,\dots,X_n)|
&\leq \biggl(1-\frac{\ff{n}{p}}{n^p}\biggr)\langle\phi,\mathbf{A}\rangle+\biggl|\frac{\ff{n}{p}}{n^p}\langle\phi,\mathbf{A}\rangle-\mathbb{E}h(X_1,\dots,X_n)\biggr|\\
&\leq
\biggl(1-\frac{\ff{n}{p}}{n^p}\biggr)\langle\phi,\mathbf{A}\rangle+\frac{|[n]^{[p]}-{\rm Inj}([p],[n])|}{n^p}\\
&<\frac{2}{n}
\end{align*}

Now remark that for every $x_1,\dots,x_n,\hat{x_i}$ it holds
$$
|h(x_1,\dots,x_i,\dots,x_n)-h(x_1,\dots,\hat{x_i},\dots,x_n)|\leq\frac{p}{n},
$$
as $p/n$ bounds the probability that an mapping from $[p]$ to $[n]$ will map some value to $i$. Thus, according to Lemma~\ref{lem:mcd} it holds for any $t>0$:
$$
{\rm Pr}(|h(X_1,\dots,X_n)-\mathbb{E}h(X_1,\dots,X_n)|\geq t)\leq 2e^{-\frac{2nt^2}{p^2}}.
$$

In particular, for $t=\epsilon-2/n$ it holds
\begin{align*}
&{\rm Pr}(|\langle\phi,\mathbf{A}\rangle - \langle\phi,\mathbf{B}_{X_1,\dots,X_n}\rangle|\geq \epsilon)\\
&\qquad\leq
{\rm Pr}(|\mathbb{E}h(X_1,\dots,X_n)-h(X_1,\dots,X_n)\rangle|\geq  \epsilon-2/n)\\
&\qquad\leq 2e^{-\frac{(\epsilon n-2)^2}{p^2n}}.
\end{align*}
\end{proof}

By union bound, we deduce that for sufficiently large $n$ there exists an $(n,L)$-sampling which has $\epsilon$-close Stone pairing with any formula with at most $p$ free variables and at most $L$ functional symbols. Precisely, we have:

\begin{corollary}
\label{cor:approx}
For every signature $\sigma$
there exists a function $R:\bbbn\times\bbbn\rightarrow\bbbn$ with the following property:

For every Borel $\sigma$-structure $\mathbf{A}$, every $\epsilon>0$ and every $L\in\bbbn$ there exists, for each $n\geq R(p,L)\,\epsilon^{-2}$ an $(n,L)$-sampling $\mathbf B$ of $\mathbf{A}$ such that for every 
formula $\phi$ with at most $p$ free variables and $L$ functional symbols it holds
$$
|\langle\phi,\mathbf{A}\rangle-\langle\phi,\mathbf{B}\rangle|<\epsilon.
$$
\end{corollary}

Hence we have:
\begin{corollary}
\label{cor:wsamp}
Every Borel $\sigma$-structure is the limit of a sequence of weighted finite $\sigma$-structures.
\end{corollary}


Note that the finite weighted structures obtained as $(n,L)$-sampling of a Borel structure $\mathbf{A}$ usually have many elements with $0$ measure.
The problem of determining whether an infinite Borel structure $\mathbf A$ is the limit of a sequence of finite unweighted  structures is much more difficult. 
Note that we have some (easy) necessary conditions on $\mathbf{A}$:
\begin{itemize}
\item the domain $A$ is uncountable and the measure $\mu_{\mathbf{A}}$ is atomless;
\item for every definable functions $f,g:A^r\rightarrow A^r$, and every definable subset $X\subseteq A^r$ of the set of fixed points of $f\circ g$, the sets $X$ and $g(X)$ have the same measure.
\end{itemize}

The second condition can be seen as a simple generalization of the {\em intrinsic mass transport principle} of Benjamini and Schramm: a graphing indeed defines a purely functional structure, with $d$ functional symbols, each interpreted as a measure preserving involution. In this case, the existence for each graphing of a sequence of bounded degree graphs having the given graphing as its limits is the difficult and well-known Aldous-Lyons conjecture \cite{Aldous2006}. 
It turns out that one of the main difficulties of this problem concerns the expansion properties of the graphing. This leads naturally to first consider a weakened version we present now (for generalized structures):

Let $\mathbf A$ be a $\sigma$-structure and let $v_1,\dots,v_p\in A$. The
{\em substructure} of $\mathbf A$ {\em generated} by $v_1,\dots,v_p$ is the $\sigma$-structure, denoted $\mathbf A[v_1,\dots,v_p]$, whose domain is the smallest subset of $A$ including $\{v_1,\dots,v_p\}$ closed by all functional symbols of $\sigma$, with the same interpretation of the relations and functions symbols as $\mathbf A$.

We shall now prove that in the case of atomless Borel tree semi-lattices, the sampling techniques can be used to build arbitrarily good finite approximations, thus to build a converging sequence of finite tree semi-lattices with the given Borel semi-lattice as a limit.

\begin{theorem}
\label{thm:uapprox}
Every atomless \mato{} $(\mathbf{T}, \mu)$ is limit of a sequence of uniform finite \mato{}s.
\end{theorem}
\begin{proof}
Let $(\mathbf{T}_N, \mu_N)$ be an $(N, 1)$-sampling of $(\mathbf{T}, \mu)$ (note that in the context of \ato{}s, and according to Remark~\ref{rem:gato}, taking $L > 1$ yields the same structures as with $L = 1$).
  Let $M = |S| - N$, that is, the number of vertices of $\mathbf{T}_N$ that were not directly sampled. 

Fix $C\in\bbbn$ and let $[C]=\{1,\dots,C\}$.
Let $(\mathbf{\hat{T}}_N, \hat\mu_N)$ be the \mato{} with elements set $\hat{T}_N\subseteq
T_n\times [C]$ defined by:
$$\hat{T}_N=\{x\in T_N: \mu_N(\{x\})=0\}\times\{1\} \cup \{x\in T_N: \mu_N(\{x\})>0\}\times [C]$$
with meet operation defined by
$$
(x,i)\wedge (y,j)=\begin{cases}
(x\wedge y, 1)&\text{if }x\wedge y\notin\{x,y\}\\
(x,i)&\text{if }x\wedge y=x\text{ and }x\neq y\\
(y,j)&\text{if }x\wedge y=y\text{ and }x\neq y\\
(x,\max(i,j))&\text{if }x=y
\end{cases}
$$
and uniform measure $\hat\mu_N$.

Informally, $(\mathbf{\hat{T}}_N, \hat\mu_N)$ is
obtained from $(\mathbf{T}_N, \mu_N)$ by replacing each of the randomly selected elements used to create $\mathbf{T}_N$ 
 with a chain on $c$ vertices, and considering a uniform measure.

Define the map 
$$\pi : \hat{T}_N \to T_N,\quad (x,i)\mapsto x.$$
Note that for every quantifier free formula $\phi$ with $p$ free variables, it holds, for every distinct $u_1,\dots,u_p\in T_N$ and every $v_1\in\pi^{-1}(u_1),\dots, v_p\in\pi^{-1}(u_p)$ it holds:
$$
\mathbf{T}_N\models \phi(u_1,\dots,u_p)
\quad \iff\quad 
\mathbf{\hat{T}}_N\models \phi(v_1,\dots,v_p).
$$

 Let $\mathbf{L}$ be a \mato{}, $p$ an integer, and $X_1, \dots, X_p$ (resp. $Y_1, \dots, Y_p$) independent random variables in $T_N$ (resp. $\hat{T}_N$) (with $X_i(\omega)=\mu_N(\omega)$, resp.
 $Y_i(\omega)=\hat\mu_N(\omega)$). Let $\phi$ be a quantifier-free formula. As for any formula $\psi$ and any structure $\mathbf A$ it holds $\langle \phi,\mathbf A\rangle-\langle\phi\Land\psi,\mathbf{A}\rangle\leq \langle\neg\psi,\mathbf{A}\rangle$, we have

 
\begin{align*}
&\biggl|\Pr\Bigl(\mathbf{T}_N\models\phi(X_1,\dots,X_p)\Bigr)-
\Pr\Bigl(\mathbf{T}_N\models\phi(X_1,\dots,X_p) \Land \lnot\bigvee_{i \neq j} (X_i = X_j) \Bigr) \biggr|\\
&\quad \leq \Pr\Bigl(\bigvee_{1\leq i <j\leq p}(X_i=X_j)\Bigr)
 \leq \binom{p}{2} \Pr\Bigl(X_1=X_2\Bigr)
 \leq \frac{p^2}{2N}
\end{align*}
 
%

Similarly,  denoting $B$ the event 
$$\bigvee_{1\leq i<p}(\pi(Y_i)=\pi(Y_j))\,\vee\,\bigvee_{i=1}^p (\mu_N(\pi(Y_i))=0)$$
 it holds

\begin{align*}
&\biggl| \Pr\Bigl(\mathbf{\hat{T}}_N\models\phi(Y_1, \dots, Y_p)\Bigr) - \Pr\Bigl(\mathbf{T}_N\models\phi(Y_1, \dots, Y_p) \Land\lnot B \Bigr) \biggr|\\
 &\quad\leq \Pr(B)\leq \frac{pM}{CN+M}+\frac{p^2}{2N}
 \leq \frac{p}{N}\Bigl(\frac{M}{C}+\frac{p}{2}\Bigr)
\end{align*}

But, denoting by $I_{A}$ the indicator function of set $A$,  it holds:

\begin{align*}
&\Pr\Bigl(\mathbf{T}_N\models\phi(X_1,\dots,X_p) \Land\lnot \bigvee_{i \neq j} X_i = X_j \Bigr)\\
&\quad= \sum_{f\in{\rm Inj}([p],T_N)} \Bigl(\prod_{i=1}^p \mu_N(\{f(i)\})\Bigr)\, I_{\phi(\mathbf{T}_n)}(f(1), \dots, f(p))\\
&\quad=\frac{1}{N^p} \sum_{f\in{\rm Inj}([p],T_N)}  I_{\phi(\mathbf{T}_n)}(f(1), \dots, f(p))\\
&\quad=\frac{1}{(CN)^p} \sum_{f\in{\rm Inj}([p],T_N)}
\sum_{g\in[C]^{[p]}}  I_{\phi(\mathbf{T}_n)}(f(1), \dots, f(p))\\
&\quad=\frac{1}{(CN)^p} \sum_{f\in{\rm Inj}([p],T_N)}
\sum_{g\in[C]^{[p]}}  I_{\phi(\hat{\mathbf{T}}_n)}((f(1),g(1)), \dots, (f(p),g(p)))\\
&\quad=\frac{(CN+M)^p}{(CN)^p} 
 \Pr\Bigl(\mathbf{\hat{T}}_N \models\phi(Y_1,\dots,Y_p) \Land\lnot B \Bigr).
\end{align*}
Thus
\begin{align*}
&\biggl|\Pr\Bigl(\mathbf{T}_N\models\phi(X_1,\dots,X_p) \Land \lnot\bigvee_{i \neq j} X_i = X_j \Bigr)-
\Pr\Bigl(\mathbf{\hat{T}}_N \models\phi(Y_1,\dots,Y_p) \Land\lnot B \Bigr)
\biggr|\\
&\quad\leq \biggl(\frac{CN+M}{CN}\biggr)^p-1\leq \frac{pM}{CN}
\end{align*}

Altogether, we have:
$$
\biggl|\Pr\Bigl(\mathbf{T}_N\models\phi(X_1,\dots,X_p)\Bigr)-
\Pr\Bigl(\mathbf{\hat{T}}_N \models\phi(Y_1,\dots,Y_p)\Bigr)
\biggr|\leq \frac{p}{N}\Bigl(p+\frac{2M}{C}\Bigr).
$$

In other words, it holds
$$
\Bigl|\langle\phi,(\mathbf{T}_N,\mu_N)\rangle-
\langle\phi,\mathbf{\hat{T}}_N\rangle
\Bigr|\leq \frac{p}{N}\Bigl(p+\frac{2M}{C}\Bigr).
$$

Together with Corollary~\ref{cor:wsamp}, we get that for
every atomless Borel tree semi-lattice $(\mathbf{T},\mu)$ ,
every $\epsilon>0$ and every $p\in\bbbn$ there exists
a finite (unweighted) tree semi-lattice $\hat{\mathbf{T}}$  such that
for every quantifier free formula $\phi$ with $p$ free variables it holds
$$
\left|\langle\phi,\mathbf{T}\rangle-\langle\phi,\hat{\mathbf{T}}\rangle\right|<\epsilon,
$$
hence if we choose $p> -\log_2\epsilon$ it holds
 ${\rm dist}(\mathbf{T},\hat{\mathbf{T}})<\epsilon +2^{-p}< 2\epsilon$.

\end{proof}

\section{Limits of Tree-Semilattices}
In this section, we focus on providing an approximation lemma for finite colored \ato{}s, which can be seen as an analog of the weak version of Szemerédi's regularity lemma. 
For the sake of simplicity, in this section, by {\ato} we always mean
a finite weighted colored {\ato}.

\subsection{Partitions of {\ato}s}
\begin{definition}
  Let $(\mathbf{T}, \mu)$ be a {\ato} and let $\epsilon > 0$. Then $v
  \in T$ is said:
  \begin{itemize}
    \item {\em $\epsilon$-light} if $\mu(T_v) < \epsilon$;
    \item {\em $\epsilon$-singular} if $\mu(T_v') \geq \epsilon$ where
      $T_v'$ is the set $T_v$ minus the sets $T_u$ for
      non-$\epsilon$-light children $u$ of $v$;
    \item {\em $\epsilon$-chaining} if $v$ is not singular and has
      exactly $1$ non-light child;
    \item {\em $\epsilon$-branching} if $v$ is not singular and has at
      least $2$ non-light children.
  \end{itemize}
  
\end{definition}

(One can easily convince themselves that every vertex of a \ato{} falls in exactly one of those categories.)

\begin{definition}
\label{def:partition}
  Let $(\mathbf{T}, \mu)$ be a \ato{}. 
  A partition $\mathcal{P}$ of
  $T$ is an {\em $\epsilon$-partition} of $(\mathbf T,\mu)$ if
  \begin{itemize}
  \item each part is of one of the following  types (for some $v\in T$, see Fig.~\ref{fig:types}):
  \begin{enumerate}
  \item\label{typ1} $P=\{v\}$,
  \item\label{typ2} $P=\{v\} \cup \bigcup_{x \in F} T_x$ for some non-empty $F\subseteq F_v$,
  \item\label{typ3} $P=\bigcup_{x \in F} T_x$ for some $F
    \subseteq F_v$ with $|F|\geq 2$,
  \item\label{typ4} $P=T_v \setminus T_w$ for some $w\in T_v$ distinct from $v$ ($w$ is called the {\em cut vertex} of $P$, and the path from $v$ to the father of $w$ is called the {\em spine} of $P$),
  \end{enumerate}
where $v$ (which is easily checked to be the infimum of $P$) is called the {\em attachement vertex} of $P$ and is denoted by $A_{\mathcal P}(P)$; 
  \item each attachment vertex of a part of type~\ref{typ3} is also the attachment vertex of some part of type~\ref{typ1} or~\ref{typ2};
  \item every part which is not a singleton has $\mu$-measure at most $\epsilon$.
  \end{itemize}
\end{definition}

\begin{figure}[ht]
\begin{center}
\includegraphics[width=\textwidth]{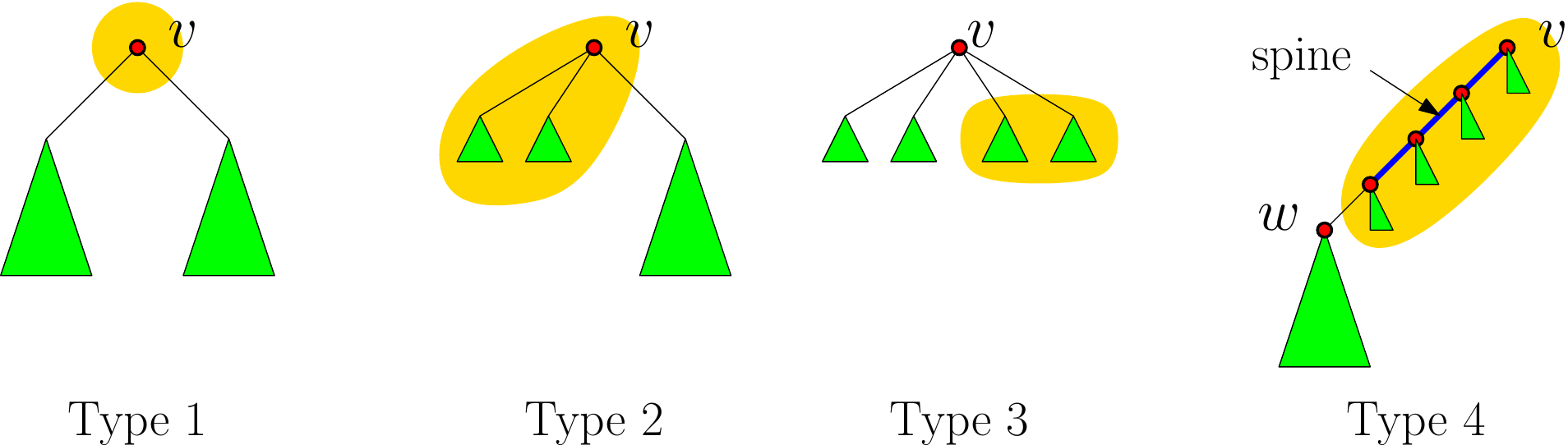}
\end{center}
\caption{Types of part in an $\epsilon$-partition}
\label{fig:types}
\end{figure}

An $\epsilon$-partition $\mathcal P$ of a {\ato} $(\mathbf{T},\mu)$ canonically defines a quotient rooted tree:

\begin{definition}
Let $\mathcal P$ be an $\epsilon$-partition of a {\ato} $(\mathbf{T},\mu)$.  The {\em quotient rooted tree} $\mathbf{T}/\mathcal{P}$ is the rooted tree with node set
$\mathcal P$, the root of which is the unique part of $\mathcal P$ that contains the minimum element of $\mathbf{T}$, where the father of
any non-root part $P$ is the part that contain the father of $A_{\mathcal P}(P)$ if $A_{\mathcal P}(P)\notin P$  (i.e. types 1,2, or 4) or $A_{\mathcal P}(P)$ itself if $A_{\mathcal P}(P)\in P$ (i.e. type 3). By abuse of notation, $\mathbf{T}/\mathcal{P}$ will also denote the \ato{} defined by the ancestor relation in the rooted tree $\mathbf{T}/\mathcal{P}$. 
\end{definition}

In several circumstances it will be handy to refer to an $\epsilon$-partition $\mathcal P$ of a {\ato} $(\mathbf{T},\mu)$ directly by means of the partition map $f:T\rightarrow \mathcal{P}$ (where $\mathcal{P}$ is meant as the vertex set of $\mathbf{T}/\mathcal{P}$).
Note that this mapping is a {\em weak homomorphism} in the sense  that for every $x,y\in T$ it holds
$f(x\wedge y)=f(x)\wedge f(y)$ except (maybe) in the case where
$f(x)=f(y)$.
The definition of  $\epsilon$-partitions can easily be transposed to provide a characterization of those weak  homomorphisms $f:\mathbf{T}\rightarrow\widetilde{\mathbf{T}}$ that define an $\epsilon$-partition.
Such mappings are called {\em $\epsilon$-partition functions}.

We now prove that every {\ato} has a small $\epsilon$-partition.

\begin{lemma}\label{lemma:epspartition}
Let $(\mathbf{T}, \mu)$ be a  \ato{} and $\epsilon > 0$. Then there exists an $\epsilon$-partition $\mathcal{P}_\epsilon(\mathbf{T}, \mu)$ of $(\mathbf{T}, \mu)$ with at most $4/\epsilon$ elements.
\end{lemma}
\begin{proof}
For sake of clarity, we construct the desired $\epsilon$-partition in two steps. First, let $\mathcal{P}$ be the partition of $T$ obtained in the following manner:
\begin{itemize}
\item for every $\epsilon$-singular vertex $v$, keep $\{v\}$ in its own part, and, for an arbitrary order on the $\epsilon$-light children of $u$, group the $\{T_u ; \text{ $u$ $\epsilon$-light child of $v$}\}$ greedily such that the $\mu$-measure of each part is maximum while remaining less than $\epsilon$.
\item for every $\epsilon$-branching or $\epsilon$-chaining vertex $v$, group $v$ with $\{T_u ; \text{ $u$ $\epsilon$-light child of $v$}\}$,

\end{itemize}
It is easily seen that $\mathcal{P}$ is indeed an $\epsilon$-partition of $(\mathbf{T}, \mu)$ (with only parts of type $1,2,3$). However, the number of parts is not bounded by $O(1/\epsilon)$.

Now, as per the definition, an $\epsilon$-chaining vertex has at most one $\epsilon$-chaining child, and the previous construction never groups the two together. Hence, we can consider chains of parts rooted at $\epsilon$-chaining vertices, the parts of which we merge greedily (starting from the closest to the root, and going from parent to child) so that each part has maximum $\mu$-measure while remaining less than $\epsilon$, thus yielding a partition $\mathcal{P}_\epsilon(\mathbf{T}, \mu)$, in which every part is either:
\begin{enumerate}
\item an $\epsilon$-singular vertex alone, 
\item a subtree of $T$ rooted at an $\epsilon$-branching vertex,
\item or $\bigcup_{x \in F} T_x$ for $F$ a set of children of some $\epsilon$-singular vertex.
\item $T_u \setminus T_v$ for an $\epsilon$-chaining vertex $u$ and a descendent $v$ of $u$,
\end{enumerate}
Note that these four categories correspond exactly to the $4$ types described in Definition \ref{def:partition}, and since by definition, all parts are indeed of $\mu$-measure at most $\epsilon$, we get an $\epsilon$-partition of $(\mathbf{T}, \mu)$. Now we need to prove that the number of parts is bounded.

First, note that there are at most $1/\epsilon$ sets of the first kind. Indeed, to each singular vertex $v$ correspond a subtree ($T_v$ minus the sets $T_u$ for non-$\epsilon$-light children $u$ of $v$) of measure at least $\epsilon$, and each of these subtrees are disjoint.

Sets of type $2$ correspond to $\epsilon$-branching vertices. Consider the tree obtained by deleting $\epsilon$-light vertices. We obtain a rooted tree, in which every chaining vertex has exactly one child, branching vertices have at least two children, and leaves are necessarily singular vertices. Therefore the number of branching vertices is at most the number of singular vertices, hence there are at most $1/\epsilon$ sets of type $2$.

Finally, note that in the greedy construction of sets of both types 3 or 4, we apply a similar principle : we have a collection of disjoint sets, each of measure at most $\epsilon$, and of total measure, say $p$, and we partition this collection by forming groups of total measure at most $\epsilon$. It is an easy observation that this can be done using at most $2p/\epsilon$ groups : one can sets greedily by decreasing order of measure -- this insures that all groups but the last have weight at least $\epsilon/2$. Moreover these groups are overall all disjoint (a non-light vertex cannot be the descendent of a light one), so the total number of parts of type $3$ and $4$ is at most $2/\epsilon$. This concludes our proof.

\end{proof}

A partition $\mathcal{P}$ is said to be a {\em refinement} of another one $\mathcal{P'}$ if each element of $\mathcal{P}$ is a a subset of an element of $\mathcal{P'}$. 

\begin{lemma}
\label{lemma:epspartition2}
Let $\epsilon' < \epsilon$, $\mathbf{T}$ be a (finite) \ato{}, and $\mathcal{P}$ be an $\epsilon$-partition of $\mathbf{T}$. Then there exists an $\epsilon'$-partition $\mathcal{P}'$ of $\mathbf{T}$ with at most $81/\epsilon'$ elements, that is a refinement of $\mathcal{P}$. Moreover, $\mathcal{P}$ induces an $\epsilon$-partition of the \ato{} $T/{\mathcal{P'}}$.
\end{lemma}
\begin{proof}
Each part not of type $3$ is a tree with total measure at most $\epsilon$, so we can apply Lemma \ref{lemma:epspartition} independetly on each of these trees. For parts of type $3$, we start by putting the attachment vertex back in the part, and then again apply Lemma \ref{lemma:epspartition}, before removing it. Thus, we obtain an $\epsilon'$-partition of $\mathbf{T}$ with at most $16/\epsilon'$ elements.
\end{proof}

\subsection{Approximations}

In the previous subsection, we defined $\epsilon$-partitions and described the quotient map and quotient tree associated to an $\epsilon$-partition. These are convenient objects to represent the partition (in particular with respects to successive refinement) but they miss some information about the measure and the colors. This is why we introduce her the concept of $f$-reduction. We first give the definition and give two easy but essential lemmas before showing how to construct the particular ''small'' reductions that will be of use to construct appoximations in the proof of our main Theorem.

\begin{definition}
Let $(\mathbf{T},\mu)$ be a {\ato} and let $f:\mathbf{T}\rightarrow\widetilde{\mathbf{T}}$ be an $\epsilon$-partition function of $(\mathbf{T},\mu)$. 

An {\em $f$-reduction} of $\mathbf{T}$ is 
a color-preserving mapping $\pi:\mathbf{T}\rightarrow\hat{\mathbf{T}}$, 
where $\hat{\mathbf{T}}$ is a {\ato}, such that
$f$ factorizes as $\hat{f}\circ\pi$, where $\hat{f}$ is an $\epsilon$-partition function of $(\hat{\mathbf{T}},\mu\circ\pi^{-1})$, and $\pi$ satisfies the property that for every $x,y\in T$ with $f(x)\neq f(y)$ it holds
\begin{align*}
\pi(x\wedge y)&\ =\ \pi(x)\wedge\pi(y)\\
x\leq y&\iff \pi(x)\leq\pi(y)
\end{align*}
Such a situation we depict by the following commutative diagram:
$$
\xy\xymatrix{
\mathbf{T}\ar[d]_f\ar[r]^{\pi}&\hat{\mathbf T}\ar[dl]^{\hat{f}}\\
\widetilde{\mathbf{T}}
}\endxy
$$
\end{definition}

\begin{lemma}
\label{lem:pi_isom}
  Let $(\mathbf{T}, \mu)$ be a
  \ato{}, let $\epsilon > 0$, let 
$f: \mathbf{T}\rightarrow\widetilde{\mathbf{T}}$ be an
  $\epsilon$-partition function of $(\mathbf{T}, \mu)$,  
  let $\pi:\mathbf{T}\rightarrow\hat{\mathbf{T}}$ be an $f$-reduction of
  $\mathbf{T}$, and let $v_1,\dots,v_p\in T$ be such that $f(v_i)\neq f(v_j)$ whenever $v_i\neq v_j$.

Then $\pi$ induces 
 an isomorphism of $\mathbf{T}\langle v_1,\dots,v_p\rangle$ and
 $\hat{\mathbf{T}}\langle \pi(v_1),\dots,\pi(v_p)\rangle$
\end{lemma}
\begin{proof}
 By assumption, for every distinct $v_i,v_j$ it holds $f(v_i)\neq f(v_j)$ thus
 $\pi(v_i\wedge v_j)=\pi(v_i)\wedge\pi(v_j)$. As $\pi$ is color preserving,
in order to prove that $\pi$ is an homorphism, we only have to check (as $\wedge$ is associative) that if $x\in \mathbf{T}\langle v_1,\dots,v_p\rangle$ and $1\leq k\leq p$ then $\pi(x\wedge v_k)=\pi(x)\wedge\pi(v_k)$.
According to Remark~\ref{rem:gato}, there exist $i,j$ such that $x=v_i\wedge v_j$. Moreover, there exists a permutation $i',j',k'$ of $i,j,k$ such that 
$v_i\wedge v_j\wedge v_k=v_{i'}\wedge v_{k'}=v_{j'}\wedge v_{k'}$.
 Hence
$\pi(v_i\wedge v_j\wedge v_k)=\pi(v_{i'}\wedge v_{k'})=\pi(v_{i'})\wedge \pi(v_{k'})$ and similarly $\pi(v_i\wedge v_j\wedge v_k)=\pi(v_{j'})\wedge \pi(v_{k'})$. Thus $\pi(x\wedge v_k)=\pi(v_i\wedge v_j\wedge v_k)=\pi(v_{i'})\wedge \pi(v_{j'})\wedge \pi(v_{k'})=\pi(v_i\wedge\pi(v_j)\wedge\pi(v_k)=\pi(x)\wedge \pi(v_k)$.

In order to prove that $\pi$ is an isomorphism, we have to check that $\pi$ is injective, that is, according to Remark~\ref{rem:gato}, that $\pi(v_i\wedge v_j)=\pi(v_k\wedge v_\ell)$ not only implies $f(v_i\wedge v_j)=f(v_k\wedge v_\ell)$, but also
implies $v_i\wedge v_j=v_k\wedge v_\ell$. 
Let $v_i,v_j,v_k,v_\ell$ be all distinct.
Then $f(v_i\wedge v_j)$ is of type $1$ or $2$, thus $v_i\wedge v_j$ and $v_k\wedge v_\ell$ coincides with  the attachment vertex of $f(v_i\wedge v_j)$ thus $v_i\wedge v_j=v_k\wedge v_\ell$.
Also, if $\pi(v_i\wedge v_j)=\pi(v_k)$ then $f(v_k)$ has type $1$ or $2$. As $\pi(v_k)\leq \pi(v_i)$ it holds $v_k\leq v_i$ hence $v_k$ is the attachment vertex of $f(v_k)$ thus $v_k=v_i\wedge v_j$.
If $\pi(v_i\wedge v_k)=\pi(v_j\wedge v_k)$ then either 
$f(v_i\wedge v_k)$ has type $1$ or $2$, in which case $v_i\wedge v_k=v_j\wedge v_k$, or $f(v_i\wedge v_k)$ has type $4$ and again $v_i\wedge v_k=v_j\wedge v_k$.
Last, if $\pi(v_i\wedge v_j)=\pi(v_j)$ then $\pi(v_j)\leq \pi(v_i)$ thus
$v_j\leq v_i$, that is $v_i\wedge v_j=v_j$.
\end{proof}

\begin{lemma}
\label{lem:redsim}
  Let $(\mathbf{T}, \mu)$ be a
  \ato{}, let $\epsilon > 0$, let 
$f: \mathbf{T}\rightarrow\widetilde{\mathbf{T}}$ be an
  $\epsilon$-partition function of $(\mathbf{T}, \mu)$,  
  let $\pi:\mathbf{T}\rightarrow\hat{\mathbf{T}}$ be an $f$-reduction of
  $\mathbf{T}$, and let $\hat{\mu}=\mu\circ\pi^{-1}$.
  
  Then for every quantifier free formula $\phi\in{\rm QF}_p$ it holds
  $$| \langle \phi, (\mathbf{T}, \mu) \rangle - \langle \phi,
  (\hat{\mathbf{T}}, \hat{\mu}) \rangle | <
  p^2\epsilon.$$
  Thus, 
  $${\rm dist}((\mathbf{T}, \mu), (\hat{\mathbf{T}}, \hat{\mu}))
  <6 \epsilon.$$
\end{lemma}
\begin{proof}
Let $\hat{f}:\hat{\mathbf{T}}\rightarrow\widetilde{\mathbf{T}}$ be the $\epsilon$-partition function associated to the $f$-reduction.
Define the sets
\begin{align*}
A&=\{(v_1,\dots,v_p)\in\phi(\mathbf{T}):\ \forall i\neq j\ (f(v_i)=f(v_j))\rightarrow (v_i=v_j)\}\\
\hat A&=\{(z_1,\dots,z_p)\in\phi(\hat{\mathbf{T}}):\ \forall i\neq j\ (\hat f(z_i)=\hat f(z_j))\rightarrow (z_i=z_j)\}
\end{align*}
According to Lemma~\ref{lem:pi_isom}, it holds $\mu^{\otimes p}(A)=\hat{\mu}^{\otimes p}(\hat A)$.
As 
\begin{align*}
&| \langle \phi, (\mathbf{T}, \mu) \rangle -\mu^{\otimes p}(A)|\\
&\qquad\leq
\mu^{\otimes p}\left(\left\lbrace (v_1,\dots,v_p)\in T:\ \exists i\neq j\ (f(v_i)=f(v_j))\wedge\neg (v_i=v_j)\right\rbrace\right)\\
&\qquad\leq \binom{p}{2}\max\{\mu(f^{-1}(p)):\ p\in\widetilde{\mathbf T}\text{ and }|f^{-1}(p)|>1\}\\
&\qquad <p^2\epsilon/2
\end{align*}
and as the same holds for $| \langle \phi, (\hat{\mathbf{T}}, \hat{\mu}) \rangle -\hat{\mu}^{\otimes p}(\hat A)|$
we deduce that $| \langle \phi, (\mathbf{T}, \mu) \rangle - \langle \phi,
  (\hat{\mathbf{T}},\hat{\mu}) \rangle | <
  p^2\epsilon$.  As $\sum_p p^2 2^{-p}=6$ the bound on the distance follows.
\end{proof}

We can now define the particular $f$-reductions, called {\em standard $f$-reductions} that we will consider to construct approximations. The fact that this construction yields indeed an $f$-reduction is proven it Lemma \ref{lemma:stdred}.

\begin{definition}
  Let $(\mathbf{T}, \mu)$ be a
  \ato{}, let $\epsilon > 0$, and let 
$f: \mathbf{T}\rightarrow\widetilde{\mathbf{T}}$ be an
  $\epsilon$-partition function of $(\mathbf{T}, \mu)$.
  
  The {\em standard $f$-reduction} $\pi:\mathbf{T}\rightarrow\hat{\mathbf{T}}$ of $(\mathbf{T}, \mu)$ is defined as follows. For each part $P$ of $\mathcal P$ we associate a rooted tree (or a rooted forest) $\mathbf Y_P$ and we define the projection $\pi$ from $T$ to the
 domain $\hat{T}=\bigcup Y_P$ of $\hat{\mathbf T}$ as follows:
\begin{itemize}
\item If $P$ is of type $1$ ($P = \{v\}$):

Then  $\mathbf Y_P$ is a single-node rooted tree,  $\pi$ maps $v$ to the only vertex in  $\mathbf Y_P$, which is assigned the same color and  weight as $v$.
 \item If $P$ is of type $2$ ($P=\{v\} \cup \bigcup_{x \in F} T_x$):
 
Then $\mathbf Y_P$ is rooted at a vertex  $a_P$ (having same color and weight as $v$) that has, for each color $\gamma$ present in $P\setminus\{v\}$, a child $b_{P,\gamma}$ of color $\gamma$ whose weight is the sum of the weights of $\gamma$-colored vertices of $P\setminus\{v\}$; the projection $\pi$ maps $v$ to $a_P$ and $u\in P\setminus\{v\}$ to
$b_{P,c(u)}$.
  \item If $P$ is of type $3$ ($P=\bigcup_{x \in F} T_x$):
  
Then $\mathbf{Y}_P$ is a set of single node rooted trees, with roots $b_{P,\gamma_1},\dots,b_{P,\gamma_\ell}$, where  $\gamma_1, \dots, \gamma_\ell$ are the colors present in $P$. The projection $\pi$ maps each vertex $u\in P$ to $b_{P,c(u)}$, and the weight of $b_{P,\gamma_i}$ is defined as the sum of the weights of the $u\in P$ with color $\gamma_i$.

  \item If $P$ is of type $4$  ($P=T_v \setminus T_w$):
  
Then   $\mathbf{Y}_P$ is a rooted caterpillar with spine $a_{P,\gamma_1},\dots,a_{P,\gamma_\ell}$ rooted at $a_{P,\gamma_1}$, where  $\gamma_1, \dots, \gamma_\ell$ are the colors present in the spine $S$ of $P$;
the projection $\pi$ maps each vertex $u\in S$ to $a_{P,c(u)}$, and the weight of $a_{P,\gamma_i}$ is defined as the sum of the weights of the $u\in S$ with color $\gamma_i$. Each vertex $a_{P,\gamma_i}$ has sons
$b_{P,\gamma_i,\rho_{i,j}}$ ($1\leq j\leq s_i\}$) where $\rho_{i,1},\dots,\rho_{i,s_i}$ are the (distinct) colors of the vertices $u\in P\setminus S$ such that
$u\wedge w$ has color $\gamma_i$. The projection $\pi$ maps each 
$u\in P\setminus S$ to $b_{P,c(u\wedge w),c(u)}$.
  \end{itemize}

 The colored {\ato} $\hat{\mathbf{T}}$ is defined by
the rooted tree $\mathbf{T}_\mathcal{P}$, which is constructed from the disjoint 
union $\bigcup_{P\in\mathcal P}\mathbf{Y}_P$ as follows:
For each non root part $P$ of $\mathcal{P}$ with father part $P'$,
 the node of $\mathbf{Y}_{P'}$ that will serve as the father of the root(s) of $\mathbf{Y}_P$ is defined as follows:
if $P'$ has type $1$ or $2$, then the father of the root of $\mathbf Y_P$ is the root of $\mathbf{Y}_{P'}$; if $P$ is of type $4$, then the father of the root of $\mathbf{Y}_{P}$ is  the maximum vertex of the spine of 
$\mathbf{Y}_{P'}$.
 \end{definition}

Before proving properties of standard $f$-reductions, let us illustrate our definitions by an example drawn on Figure \ref{fig:example1}. The tree $\mathbf{T}$ has $80$ vertices with uniform weight $1/80$ and suppose $\epsilon=1/5$.  
\begin{figure}[ht]
\begin{center}
\includegraphics[height=5.5cm]{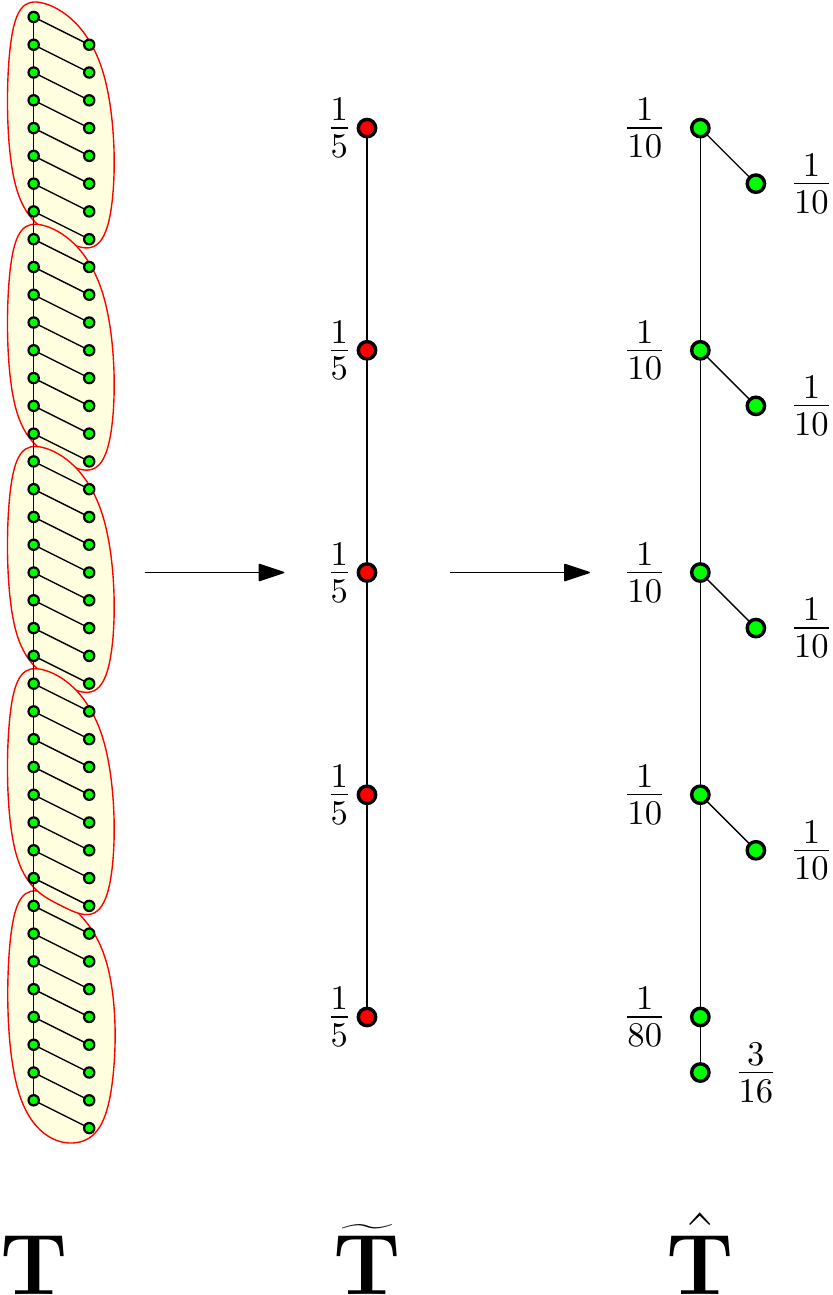}
\end{center}
\caption{Example of ${\mathbf{T}},\widetilde{\mathbf{T}},\hat{\mathbf{T}}$.}
\label{fig:example1}
\end{figure}
Consider for instance the formula 
$\phi: (x_1\wedge x_2\neq x_1) \Land (x_1\wedge x_2\neq x_2)$. Then
$\langle\phi,\mathbf{T}\rangle=0.4875$ and
$\langle\phi,\hat{\mathbf{T}}\rangle=0.4$ 
are close, but $\langle\phi,\widetilde{\mathbf{T}}\rangle=0$.

\begin{lemma}
\label{lemma:stdred}
  Let $(\mathbf{T}, \mu)$ be a
  \ato{}, let $\epsilon > 0$, and let 
$f: \mathbf{T}\rightarrow\widetilde{\mathbf{T}}$ be an
  $\epsilon$-partition function of $(\mathbf{T}, \mu)$.
  
  Then the standard $f$-reduction 
 $\pi:\mathbf{T}\rightarrow\hat{\mathbf{T}}$ of
  $\mathbf{T}$  is an $f$-reduction and
  $$|\hat{T}|\leq C(C+1) |\widetilde{T}|,$$
   where 
  $C$ denotes the number of colors for elements of $T$.
\end{lemma}
\begin{proof}
Let $\hat{\mu}=\mu\circ\pi^{-1}$.
As $\pi(x)=\pi(y)$ imply $f(x)=f(y)$ there exists a mapping 
$\hat{f}:\hat{\mathbf{T}}\rightarrow\widetilde{\mathbf{T}}$ such that
$f=\hat{f}\circ \pi$. Moreover, it is easily checked that $\hat{f}$ is an 
$\epsilon$-partition function of $(\hat{\mathbf{T}},\hat{\mu})$.

For every $u,v$ in $T$ with $f(u)\neq f(v)$,
 the element
$\pi(u)\wedge\pi(v)$ belongs to the part
 $Y_{\pi(u\wedge v)}$, which has the same type as the part $P$
 of $u\wedge v$.
If $P$ has type $1$ or $2$ then $u\wedge v=A_{\mathcal{P}}(P)$ and $\pi(u)\wedge \pi(v)=A_{\hat{\mathcal{P}}}(Y_{\pi(u\wedge v)})=\pi(A_{\mathcal{P}}(P))=\pi(u\wedge v)$. If $P$ has type $3$ then
one of $u$ and $v$ has to belong to $P$ (as $P$ has only one child) and the other one has to belong to a part greater than $P$.
Without loss of generality assume $u>v$ and $v\in P$. Then $u\wedge v=w\wedge v$, where $w$ denotes the cut-vertex of $P$.
Similarly, we have $\pi(u)\wedge \pi(v)=\hat w\wedge \pi(v)$, where $\hat w$ denotes the cut-vertex of $Y_P$. By construction, it holds
$\pi(w\wedge v)=\hat w\wedge \pi(v)$ hence $\pi(u\wedge v)=\pi(u)\wedge\pi(v)$. It follows that $\pi$ is an $f$-reduction.
 \end{proof}



%

\begin{lemma}
\label{lem:compose}
  Let $(\mathbf{T}, \mu)$ be a
  \ato{}, let $\epsilon > 0$, let 
$f_1: \mathbf{T}\rightarrow\widetilde{\mathbf{T}}_1$
(resp. $f_2: \mathbf{T}\rightarrow\widetilde{\mathbf{T}}_2$)
 be  an $\epsilon_1$-partition function
 (resp. an $\epsilon_2$-partition function) of $(\mathbf{T}, \mu)$.
 
 Let $\pi_1:\mathbf{T}\rightarrow\hat{\mathbf{T}}_1$ 
(resp. $\pi_2  :\mathbf{T}\rightarrow\hat{\mathbf{T}}_2$)
  be the standard $f_1$-reduction (resp. the standard $f_2$-reduction) of
  $\mathbf{T}$ of $\mathbf{T}$.
  
 If there exists a mapping $g:\widetilde{\mathbf{T}}_2\rightarrow\widetilde{\mathbf{T}}_1$ (i.e. if partition $f_2$ refines partition $f_1$)
 then the standard $(g\circ\hat{f}_2)$-reduction $p$ of $\mathbf{T}$
 maps $\hat{\mathbf{T}}_2$ onto $\hat{\mathbf{T}}_1$
 in such a way that the following diagram commutes:
  
$$
\xy\xymatrix@C=50pt{
\mathbf{T}\ar[d]^(.4){f_2}\ar@/_3ex/[dd]_{f_1}
\ar[r]^(.6){\pi_2}\ar@/^3ex/[rr]^{\pi_1}&\hat{\mathbf T}_2\ar[dl]_{\hat{f}_2}\ar[r]^(.3){p}\ar[ddl]|(.47){g\circ\hat{f}_2}
&\hat{\mathbf T}_1\ar[ddll]^{\hat{f}_1}\\
\widetilde{\mathbf{T}}_2\ar[d]^(.4)g\\
\widetilde{\mathbf{T}}_1
}\endxy
$$
\end{lemma}
\begin{proof}
The tedious but easy check is omitted here.
\end{proof}

\subsection{Limits of Tree-Semilattices}

\begin{theorem}
\label{thm:limwato}
Every QF-convergent sequence $(\mathbf{T}_n, \mu_n)$  of  \ato{}s has
a {\mato} limit  $(\hat{\mathbf{T}},\hat{\mu})$.
\end{theorem}
\begin{proof}
We consider a decreasing sequence $\epsilon_1>\epsilon_2>\dots$ of positive reals with limit $0$.
Using Lemma \ref{lemma:epspartition} and Lemma \ref{lemma:epspartition2}, we can construct a sequence $\mathcal{P}_{n,1},\mathcal{P}_{n,2},\dots$ of partitions such that $\mathcal{P}_{n,i}$ is an $\epsilon_i$-partition of $(\mathbf{T}_n, \mu_n)$ with at most $16/\epsilon_i$ vertices and $\mathcal{P}_{n,i+1}$ is a refinement of $\mathcal{P}_{n,i}$. 

For every fixed $i$ there are only a finite number of trees on $16/\epsilon_i$ vertices, so up to extracting subsequences (we do it sequentially starting with $i=1$), we can assume that all the $\epsilon_i$-partition functions corresponding to the $(\mathcal{P}_{n,i})_{n\in \bbbn}$ have the same target tree $\widetilde{\mathbf{T}_i}$.

And in fact the same finiteness argument is true also for standard reductions, so we can also assume that all the standard 
$f_{n,i}$-reductions $\pi_{n,i}$ have the same target 
$\hat{\mathbf T_i}$ (and partition map $\hat{f}_{i}$), and by a compactness argument we can also assume that
for each $i$ the measures $\hat{\mu}_{n,i}$ converge weakly to some measure $\hat{\mu}_i$.
We further denote by $g_{i}$ the map from $\widetilde{\mathbf{T}}_{i+1}$ to
$\widetilde{\mathbf{T}}_{i}$ witnessing that $\mathcal{P}_{i+1}$ is a refinement of $\mathcal P_i$, by $h_{i}$ the composition $g_{i}\circ \hat{f}_{i+1}$, and by $p_i:\hat{\mathbf{T}}_{i+1}\rightarrow\hat{\mathbf{T}}_i$ the standard $h_{i}$-reduction of $\hat{\mathbf{T}}_{i+1}$.

We define the set $\hat{T}$ as
$$
\hat{T}=\Bigl\{u\in\prod_{i\in\bbbn}\hat{T}_i:\quad (\forall n\in\bbbn)\  p_n(u_{n+1})=u_n\Bigr\},
$$
and we denote by $\zeta_i:\hat{T}\rightarrow\hat{T_i}$ the projection to the $n$th coordinate. Then we have the following diagram:

$$
\xy\xymatrix@C=20pt@R=20pt{
\hat{\mathbf{T}}
\ar@{..>}[drr]|{\zeta_{i+1}}\ar@{..>}@/^2ex/[drrrr]|{\zeta_i}\ar@{..>}@/^3.5ex/[drrrrrr]|{\zeta_{i-1}}\\
	&\ar[r]&\hat{\mathbf T}_{i+1}\ar[dd]|{\hat{f}_{i+1}}\ar[ddrr]|{h_{i}}\ar[rr]_{p_i}
	&&\hat{\mathbf T}_i\ar[dd]|{\hat{f}_i}\ar[ddrr]|{h_{i-1}}\ar[rr]_{p_{i-1}}
	&&\hat{\mathbf T}_{i-1}\ar[dd]|{\hat{f}_{i-1}}\ar[dr]\ar[r]
	&&\\
\dots&\ar[dr]&&&&&&&\dots\\
	&\ar[r]
	&\widetilde{\mathbf{T}}_{i+1}\ar[rr]_{g_i}
	&&\widetilde{\mathbf{T}}_{i}\ar[rr]_{g_{i-1}}
	&&\widetilde{\mathbf{T}}_{i-1}\ar[r]
	&&\\
}\endxy
$$

Note that $\hat{T}$ (with product topology) is a Cantor space, hence (with Borel $\sigma$-algebra) is a standard Borel space. 
We define the color $c(u)$ of $u\in\hat{T}$ as the color of $u_1$.
As all mappings $p_n$ are color-preserving, it follows that $c(u)$ is the color common to all of the $u_n$.

For $u,v\in\hat{T}$,and $n\in \bbbn$, we note $u_n \sim v_n$ if :
\begin{itemize}
\item$\hat{f}_n(u_n)=\hat{f}_n(v_n)$ 
\item either $u_n$ and $v_n$ are equal or $\hat{f}_n(v_n)$ is of type 4 and they are brothers (meaning they both don't belong to the spine and both are children of their infimum $u_n \wedge v_n$).
\end{itemize}
We then say that $u$ and $v$ are {\em similar} and we note $u\sim v$ if for every integer $n$ $u_n \sim v_n$. Note that $\sim$ is an equivalence relation on $\hat{T}$, and that
$$(u=v)\quad\iff\quad (u\sim v)\wedge(c(u)=c(v))$$

We define the mapping $\wedge:\hat{T}\times \hat{T}\rightarrow\hat{T}$ as follows:
Let $u,v\in\hat{T}$.
If $u\sim v$ then $u\wedge v$ is the one of $u,v$ with smaller color;
otherwise, denoting $a$ an integer such that either $\hat{f}_a(u_a)\neq \hat{f}_a(v_a)$ or $u_a$ and $v_a$ are neither equal or brothers, we define $(u\wedge v)_n=u_n\wedge v_n$ if $n\geq a$, and (inductively)
$(u\wedge v)_n=p_n((u\wedge v)_{n+1})$ if $n<a$.

The mapping $\wedge$ is clearly symmetric, and it is easily checked to be measurable and associative. Let $x,y,z\in\hat{T}$.
Assume there exists $i,j,k\in\bbbn$ such that
$x_i\not\sim  y_i$, $y_j\not\sim z_j$, and $x_k\not\sim y_k$. Let $\ell=\max(i,j,k)$. Then for every $n\geq\ell$ it holds $|\{(x\wedge y)_n,(x\wedge z)_n,(y\wedge z)_n\}|\leq 2$. Hence either for every $n\in\bbbn$ it holds
$|\{(x\wedge y)_n,(x\wedge z)_n,(y\wedge z)_n\}|=1$ (thus $x\wedge y=x\wedge z=y\wedge z$) or $|\{(x\wedge y)_n,(x\wedge z)_n,(y\wedge z)_n\}|=2$ for sufficiently large $n$, which implies  $|\{x\wedge y,x\wedge z,y\wedge z\}|=2$.

Otherwise, assume that there exists $i\in\bbbn$ such that
$x_i\not\sim  y_i$, but that $y_n\sim z_n$ holds for every $n\in\bbbn$. 
Then it is easily checked that for $n>i$ either 
it holds $x_n\wedge y_n=x_n\wedge z_n$, or that it holds
$x_n\wedge y_n=y_n$ and $x_n\wedge z_n=z_n$. In both cases we
have $|\{x\wedge y,x\wedge z,y\wedge z\}|\leq 2$.

Otherwise, $x_n\sim y_n$, $y_n\sim  z_n$ and $x_n\sim z_n$ hold for every $n\in\bbbn$. Then obviously $|\{x\wedge y,x\wedge z,y\wedge z\}|\leq 2$.

Altogether, according to Remark~\ref{rem:gato}, it follows that $\hat{\mathbf{T}}$ is a {\ato}.

For $i\in\bbbn$ we consider an arbitrary  probability measure 
$\hat{\lambda}_i$ on $\hat{T}$ with $\hat{\lambda}_i(\zeta_i^{-1}(x))= \hat{\mu}_{i}(\{x\})$ for every $x\in \hat{T}_i$. 
As all of the $p_n$ are measure preserving, we have
that for every $n> i$ and every $x\in \hat{T}_i$ it holds
\begin{align*}
\hat{\lambda}_n(\zeta_i^{-1}(x))&=
\hat{\lambda}_n\circ\zeta_n^{-1}\circ p_{n-1}^{-1}\circ\dots\circ p_i^{-1}(x)\\
&=\hat{\mu}_{n}\circ p_{n-1}^{-1}\circ\dots\circ p_i^{-1}(x)\\
&=\hat{\mu}_{i}(\{x\})\\
&=\hat{\lambda}_i(\zeta_i^{-1}(x))
\end{align*}

As the sets $\zeta_i^{-1}(\{x\})$ (for $i\in\bbbn$ and $x\in\hat{T}_i$) generate the open sets of $\hat{T}$,
it follows  that measures $\hat{\lambda}_i$ converge weakly to a probability measure $\hat{\mu}$ as $i$ grows to infinity.

Let $\phi$ be a quantifier-free formula with $p$ free variables.
Let $u^1,\dots,u^p$ be random elements of $\hat{\mathbf{T}}$ chosen independently with respect to probability measure $\hat{\mu}$.
With probability at least $1-\binom{p}{2}\epsilon_q$ no two distinct $u^i$'s have the same $\zeta_q$ projection and thus
$\hat{\mathbf{T}}\langle u^1,\dots,u^p\rangle$ is isomorphic to
$\hat{\mathbf{T}}_q\langle u^1_q,\dots,u^p_q\rangle$.
As in Lemma~\ref{lem:redsim}, we deduce that 
$$
|\langle\phi,(\hat{\mathbf{T}}, \hat{\mu})\rangle -
\langle\phi,(\hat{\mathbf{T}}_q, \hat{\mu}_q)\rangle|<p^2\epsilon_q/2.
$$
Moreover, as $n$ grows to infinity, $\hat{\mu}_{n,q}$ converges weakly to $\hat{\mu}_q$. It follows that there exists $n_0$ such that for every 
$n\geq n_0$ it holds
$$
|\langle\phi,(\hat{\mathbf{T}}_q, \hat{\mu}_q)\rangle
-\langle\phi,(\hat{\mathbf{T}}_q, \hat{\mu}_{n,q})\rangle|<\epsilon_q.
$$
Also, according to Lemma~\ref{lem:redsim}, it holds
$$
|\langle\phi,({\mathbf{T}}_n, \mu_n)\rangle
-\langle\phi,(\hat{\mathbf{T}}_q, \hat{\mu}_{n,q})\rangle|<p^2\epsilon_q/2.
$$
Thus it holds
$$
|\langle\phi,({\mathbf{T}}_n, \mu_n)\rangle
-\langle\phi,(\hat{\mathbf{T}}, \hat{\mu})\rangle|<(p^2+1)\epsilon_q.
$$

Hence $(\hat{\mathbf{T}},\hat{\mu})$ is a {\mato} limit
of the sequence $({\mathbf{T}}_n, \mu_n)_{n\in\bbbn}$.
\end{proof}

\begin{corollary}
\label{cor:limato}
Let $(\mathbf{T}_n, \mu_n)$ be a QF-converging sequence of uniform (finite) \mato{}s. Then there exists an atom-less limit $(\mathbf{T}, \mu)$.

\end{corollary}

\section{Applications: $m$-partite cographs}

A {\em cograph}, or {\em complement-reducible} graph, is a
graph that can be generated from $K_1$ by complementations and disjoint unions.
Some generalizations of cographs have been proposed;
e.g., bi-cographs \cite{Giakoumakis1997} or $k$-cographs \cite{Hung2011}.

A well known characterization of cographs is the following  \cite{Corneil1981}.
\begin{theorem}
A graph $G$ is a cograph if and only if there exists a rooted tree $T$ (called
{\em cotree}), whose set of leaves is $V(G)$, whose internal nodes are colored $0$ or $1$, and such that two vertices $u,v$ are adjacent in $G$ if and only if 
the color of their lowest common ancestor is $1$.
\end{theorem}

The following generalization was proposed in \cite{Ganian2012,Ganian2013}:

\begin{definition}[$m$-partite cograph]
\label{def:mpartite-cograph}
An {\em$m$-partite cograph} is a graph that admits an
{\em $m$-partite cotree} representation, that is a rooted tree $T$ such that
\begin{itemize}
\item
the leaves of $T$ are the vertices of $G$, and are colored by a label from
$\{1,\dots,m\}$,
\item
the internal nodes $v$ of $T$ are assigned symmetric functions
$f_v:\{1,\dots,m\}\times\{1,\dots, m\}\rightarrow \{0,1\}$ with the property
that two vertices $x$ and $y$ of $G$ with respective colours $i$ and $j$ are
adjacent iff their least common ancestor $v$ in $T$ has $f_v(i,j)=1$.
\end{itemize}
\end{definition}
(This formal definition is equivalent to the intuitive one given in Section~\ref{sec:res}.) Note that cographs are exactly $1$-partite cographs.

The class of cographs  
is random-free, as noticed
 by Janson~\cite{Janson2013}. The proof is based on the following  characterization of random-free hereditary classes of graphs given by Lov\'asz and Szegedy~\cite{Lovasz2010}. (Recall that a class $\mathcal C$ is {\em hereditary} is every induced subgraph of a graph in $\mathcal C$ is in $\mathcal C$.)

\begin{theorem}
\label{thm:vc}
A hereditary class of graphs $\mathcal C$ is random-free if and only if
there exists a bipartite graph $F$ with bipartition $(V_1,V_2)$ such that no graph obtained from $F$ by adding edges within $V_1$ and $V_2$ is in 
$\mathcal{C}$.
\end{theorem}

It is well known that a cograph is a graph which does not contain the path $P_4$ on $4$ vertices as an induced subgraph. Similarly, the class of $m$-partite 
cographs is a (well quasi-ordered) hereditary class of graphs defined by a finite number of forbidden induced subgraphs $\mathcal{F}_m$, one of which is the path $P_{3(2^m-1)+1}$ of length $3(2^m-1)$ \cite{Ganian2012,Ganian2013}. This allows us to extend the result of Janson:

\begin{lemma}
\label{lem:mcorf}
For every integer $m$, the class of all $m$-partite cographs is hereditary and random-free.
\end{lemma}
\begin{proof}
That the class of $m$-partite cographs is hereditary is immediate from the existence of an $m$-partite representation, as deleting leaves in the representation gives a representation of the corresponding induced subgraph.

The path $P$ of length $3(2^{2m}-1)$ is bipartite. Let $(V_1,V_2)$ be a bipartition of its vertex set. Assume for contradiction that a graph $H$ obtained from $P$ by adding edges within $V_1$ and $V_2$ is an $m$-partite cograph. Let $Y$ be an $m$-partite cotree representation of $H$, and assume the laves are colored $1,\dots,m$. Recolor each leaf in $V_1$ by adding $m$ to its color ($c$ becomes $c+m$). For every internal node $v$, the color assigned to $v$ corresponds to a function $f_v:\{1,\dots,m\}\times\{1,\dots,m\}\rightarrow\{0,1\}$. Replace $f_v$ by the function $\hat{f}_v:\{1,\dots,2m\}\times\{1,\dots,2m\}\rightarrow\{0,1\}$ defined as follows: for every $1\leq i,j\leq m$, let
\begin{align*}
\hat{f}_v(i,j)=\hat{f}_v(m+i,m+j)&=0\\
\hat{f}_v(m+i,j)=\hat{f}_v(i,m+j)&=f_v(i,j).
\end{align*}
Then it is easily checked that we just constructed a $2m$-cotree representation of $P$, what contradicts the fact that $P$ is not a $2m$-partite cograph. It follows that no graph $H$ obtained from $P$ by adding edges within $V_1$ and $V_2$ is an $m$-partite cograph.
Hence, according to Theorem~\ref{thm:vc}, the class of all $m$-partite cographs is random-free.
\end{proof}

It follows from Lemma~\ref{lem:mcorf} and results in \cite{Lovasz2010} that every left-convergent sequence of $m$-partite cographs has a limit which is equivalent to a Borel graph excluding (as induced subgraphs) every graph in $\mathcal{F}_m$. Thus it holds
\begin{proposition}
\label{prop:excl}
A Borel graph is the QF-limit of a sequence of finite $m$-partite cographs if and only if it excludes all the (finite) graphs in the (finite) set $\mathcal{F}_m$ as induced subgraphs.
\end{proposition}

In order to deal more easily with QF-limits, we consider the weighted colored {\ato} $(\mathbf T,\mu)$ corresponding to the rooted tree $T$, where $\mu$ is null on internal vertices of $T$, and uniform on the leaves of $T$. The  the {\em interpretation} $\mathsf{I}(\mathbf{T},\mu)$ of a colored {\mato}
 $(\mathbf{T},\mu)$ is the Borel graph $(G,\nu)$ whose vertex set is the support of $\mu$,  where $\nu$ is the restriction of $\mu$ to its support, where $x$ and $y$ are adjacent if $\Phi(x,y)$ holds, where $\Phi$ is the quantifier-free formula asserting that for some $1\leq i,j\leq m$ and $f\in \{0,1\}^{\{1,\dots,m\}\times\{1,\dots,m\}}$ with $f(i,j)=1$ it holds that  $x$ has color $i$, $y$ has color $j$, and $x\wedge y$ has color $f$. Note that $\nu$ is a probability measure on the vertex set of $G$ as the domain of $\mathbf{T}$ is a standard Borel space (from what follows that $\mu({\rm Supp}(\mu))=1$).
 
  Is is easily checked that if $(\mathbf{T},\mu)$
 is the finite weighted colored {\ato} $(\mathbf T,\mu)$ corresponding to the $m$-partite cotree representation of a finite $m$-partite cograph $G$ then $\mathsf{I}(\mathbf{T},\mu)=G$.

It is easily checked  that that for every quantifier free formula $\psi$ there exists a quantifier free formula $\psi^*$ such that for every colored {\mato}
$(\mathbf{T},\mu)$ it holds
$$
\langle\psi,\mathsf{I}(\mathbf{T},\mu)\rangle=\langle\psi^*,(\mathbf{T},\mu)\rangle.
$$

\begin{theorem}
\label{thm:limco}
Limits of $m$-partite cographs are exactly interpretations by 
$\mathsf I$ of atomless colored \mato{}s.
\end{theorem}
\begin{proof}
Let $(G_n)_{n\in\bbbn}$ be a left-convergent (hence QF-convergent) sequence of $m$-partite cographs. Let $(\mathbf T_n,\mu_n)$ be a colored weighted {\ato} such that $G_n=\mathsf{I}(\mathbf T_n,\mu_n)$.
By compactness, there exists an increasing function $f:\bbbn\rightarrow\bbbn$ such that the sequence $(\mathbf T_{f(n)},\mu_{f(n)})_{n\in\bbbn}$ is ${\rm QF}$-convergent. Let $(\hat{\mathbf T},\mu)$ be a ${\rm QF}$ limit {\mato} for the sequence $(\mathbf T_{f(n)},\mu_{f(n)})_{n\in\bbbn}$. 
Note that $\mu$ has no atom as $\lim_{n\rightarrow\infty}\sup_{v\in T_n}\mu_n(\{v\})=0$.
Define the Borel graph $(G,\nu)$ as $G=\mathsf{I}(\hat{\mathbf T},\mu)$.  For every quantifier free formula $\psi$, it holds
\begin{align*}
\lim_{n\rightarrow\infty}\langle\psi, G_n\rangle
&=\lim_{n\rightarrow\infty}\langle\psi, G_{f(n)}\rangle\\
&=\lim_{n\rightarrow\infty}\langle\psi, \mathsf{I}(\mathbf{T}_{f(n)},\mu_{f(n)})\rangle\\
&=\lim_{n\rightarrow\infty}\langle\psi^*,(\mathbf{T}_{f(n)},\mu_{f(n)})\rangle\\
&=\langle\psi^*,(\mathbf{T},\mu)\rangle\\
&=\langle\psi,\mathsf{I}(\mathbf{T},\mu)\rangle
\end{align*}
It follows that the Borel graph $\mathsf{I}(\mathbf{T},\mu)$ is the QF-limit of the sequence $(G_n)_{n\in\bbbn}$. To this Borel graph corresponds a random-free graphon, which is thus the left limit 
of the sequence $(G_n)_{n\in\bbbn}$.

Conversely, if $G=\mathsf{I}(\hat{\mathbf T},\mu)$ is the interpretation of an atomless colored {\mato} $(\mathbf{T},\mu)$,
then, using a construction analogous to the one use in the proof of Theorem~\ref{cor:approx}, one gets that $(\mathbf{T},\mu)$ is the QF-limit of a sequence of  finite weighted colored {\ato}s $(\mathbf{T}_n,\mu_n)$, where $\mu_n$ is uniform on its support. Then for every integer $n$ the interpretation $\mathsf{I}(\mathbf{T}_n,\mu_n)$ is a finite $m$-partite cograph (with uniform measure on its vertex set) and,
as above, $G$ is the limit of the sequence $(G_n)_{n\in\bbbn}$.
\end{proof}

Thus, according to Proposition~\ref{prop:excl}, we get the following extension of the characterization of finite $m$-partite cographs:

\begin{theorem}
\label{thm:char}
For an atomless Borel graph $G$ the following conditions are equivalent:
\begin{enumerate}
\item $G$ is the QF-limit of a sequence of finite $m$-partite cographs;
\item $G$ is equivalent to a random-free graphon $W$ that is the left limit of a sequence of finite $m$-partite cographs;
\item $G$ excludes all graphs in $\mathcal{F}_m$ as induced subgraphs;
\item $G$ is the interpretation by 
$\mathsf I$ of an atomless colored {\mato}; 
\end{enumerate}
\end{theorem}

\providecommand{\noopsort}[1]{}\providecommand{\noopsort}[1]{}
\providecommand{\bysame}{\leavevmode\hbox to3em{\hrulefill}\thinspace}
\providecommand{\MR}{\relax\ifhmode\unskip\space\fi MR }
\providecommand{\MRhref}[2]{%
  \href{http://www.ams.org/mathscinet-getitem?mr=#1}{#2}
}
\providecommand{\href}[2]{#2}

\end{document}